\documentclass{amsart}

\usepackage{amsmath,amssymb}
\usepackage{amsthm}
\usepackage{mathtools}

\renewcommand{\Re}{\operatorname{Re}}

\renewcommand{\tan}{\operatorname{tg}}

\DeclarePairedDelimiterX\innerp[2]{\langle}{\rangle}{#1, #2}
\DeclarePairedDelimiterX\ccint[2]{\lbrack}{\rbrack}{#1, #2}

\newtheorem{conjecture}{Conjecture}
\newtheorem{theorem}{Theorem}
\newtheorem{lemma}{Lemma}

\author{Petar Melentijevi\'c}
\thanks{The author is partially supported by MPNTR grant 174017, Serbia}
\title[Khavinson conjecture]{A proof of the Khavinson conjecture }

\begin{document}
\begin{abstract}
  This paper deals with an extremal problem for bounded harmonic functions in
  the unit ball $\mathbb{B}^n.$ We solve the Khavinson conjecture in $\mathbb{R}^3,$
  an intriguing open question since 1992 posed by D. Khavinson, later considered in a general context by Kresin and Maz'ya. Precisely, we obtain the following inequality:
  $$|\nabla u(x)|\leq \frac{1}{\rho^2}\bigg({\frac{(1+\frac{1}{3}\rho^2)^{\frac{3}{2}}}{1-\rho^2}-1}\bigg) \sup_{|y|<1} |u(y)|, $$
  with $\rho=|x|,$ thus sharpening the previously known with $|\langle \nabla  u(x),n_{x} \rangle |$ instead of $|\nabla u(x)|, $ where $n_{x}=\frac{x}{|x|}.$
\end{abstract}
\maketitle

\section{Introduction}
 
\subsection{Gradient estimates for harmonic and analytic functions.}
Sharp estimates of harmonic functions are important at many places in physics.
According to \cite{ProtterWeinberger} by Protter and Weinberger, these problems
arose naturally in the theory of hydrodynamics of ideal or the viscous
incompressible fluids, elasticity theory, electrostatics and others.
 
Many of such sharp estimates are known. We recall here some of them.
 
First, in the mentioned book of Protter and Weinberger there is the following
estimate for the absolute value of the gradient of a harmonic function:
\begin{equation}
\label{eq:1}
|\nabla u(x)|\leq \frac{n\omega_{n-1}}{(n-1)\omega_{n}d(x)} \operatorname{osc}_{\Omega}(u),
\end{equation}
where $u$ is harmonic function in $\Omega,$ $\omega_{n}$ is the area of the unit
sphere $\mathbb{S}^{n-1}=\partial\mathbb{B}^n$, $\operatorname{osc}(u)$ is the
oscillation of $u$ in $\Omega,$ while $d(x)$ denotes the distance of $x \in
\Omega$ from $\partial\Omega.$ The inequality \eqref{eq:1} is a consequence of the next
best-constant inequality
\begin{equation}
\label{eq:2}
|\nabla u(0)|\leq \frac{2n\omega_{n-1}}{(n-1)\omega_{n}R}\sup_{|x|<R} |u(x)|,
\end{equation}
see Khavinson \cite{Khavinson}, Burgeth \cite{Burgeth}.
 
Some inequalities from analytic function theory can also be restated as
inequalities for harmonic functions. Such are the so-called real-part theorems,
with some characteristics of the real-parts of a function as the majorant. It is
the case with Hadamard real-part theorem \cite{Hadamard} and the whole
collection of similar inequalities in \cite{KresinMazya3}. Also, some pointwise
sharp estimates for modulus of derivatives of analytic functions can be found in
\cite{MacintyreRogosinski}. We will mention that:
$$
|f'(z)|\leq \frac{4}{\pi}\frac{1}{1-|z|^2} \sup_{|w|<1}|\Re f(w)|,
$$
for analytic functions can be also seen as
$$
|\nabla u(z)|\leq  \frac{4}{\pi }\frac{1}{1-|z|^2} \sup_{|w|<1}|u(w)|
$$
for a harmonic function in the unit disk $\mathbb{D}:=\{z \in \mathbb{C}:
|z|<1\}.$
  The last classical result is improved by Kalaj and Vuorinen in
\cite{KalajVuorinen}; their form of this inequality is
\begin{equation}
\label{eq:3}
 |\nabla u(z)|\leq  \frac{4}{\pi}\frac{1-|u(z)|^2}{1-|z|^2}, 
\end{equation}
for real-valued harmonic functions with $|u(z)|<1$ for every $z \in \mathbb{D}.$
The inequality \eqref{eq:3} is equivalent to the fact that harmonic functions with the
assumed properties are Lipschitz with constant $\frac{4}{\pi}$ with respect to
hyperbolic metric. Mateljević in \cite{Mateljevic} showed that this result can
be concluded from Ahlfors-Schwarz lemma. Let us mention that similar results for
harmonic functions and hyperbolic metric are contained in papers
\cite{Markovic2},\cite{Colonna} and \cite{Melentijevic}. Also, some sharp
inequalities for harmonic functions are given in \cite{KalajMarkovic1} and
\cite{KalajMarkovic2}.

\subsection{The Khavinson problem}
In his paper \cite{Khavinson} from 1992, Dmitriy Khavinson found the sharp
pointwise constant in the estimate for the absolute value of the radial
derivative of a bounded harmonic function in the unit ball $\mathbb{B}^3:=\{x \in \mathbb{R}^3: |x|<1\}.$
  
In a private communication with Vladimir Maz'ya and Gershon Kresin he said that he
believed that the stronger inequality holds for the modulus of the gradient of a
bounded harmonic function in place of the radial derivative.
 
To give the precise statement of the problem, we introduce some notation. We
consider bounded harmonic functions in $\mathbb{B}^n$---it is common to denote
this function space by $h^{\infty},$ see
\cite{AxlerBourdonRamey},\cite{Pavlovic}. For every $l \in \mathbb{S}^{n-1}$ by
$C(x,l)$ we denote the best constant in the next inequality for the derivative
of $u \in h^{\infty}$ at $x \in \mathbb{B}^n$ in the direction $v$: 
\begin{equation}
\label{eq:4}
|\langle \nabla u(x), l \rangle| \leq C(x,l) \sup_{|y|<1}|u(y)|,
\end{equation}
while for the appropriate constant for the modulus of the gradient we use
$C(x):$ 
\begin{equation} 
\label{eq:5}
|\nabla u(x)| \leq C(x) \sup_{|y|<1}|u(y)|.
\end{equation}
 
Since 
$$|\nabla u(x)|=\sup_{|l|=1}|\langle \nabla u(x), l \rangle|,$$
we clearly have that
\begin{equation}
\label{eq:6}
C(x)=\sup_{|l|=1}C(x,l).
\end{equation}
We are especially interested for radial direction, which is, for $x \in
\mathbb{B}^n,$ defined as $n_{x}=\frac{1}{|x|}x,$ where
$|x|=\sqrt{x_1^2+x_2^2+\dots+x_n^2} $ is the norm of $x.$
 
In their paper \cite{KresinMazya1}, Kresin and Maz'ya posed the generalized
Khavinson problem for bounded harmonic functions in the unit ball
$\mathbb{B}^n,$ as:
 
\begin{conjecture}[\cite{KresinMazya1}, \cite{Khavinson}]
 $$	C(x)=C(x,n_{x}).$$
\end{conjecture}
  
In the same paper, they obtain the sharp inequalities for the radial and
tangential derivatives of such functions and solved the analogous problem for
harmonic functions with the $L^p$ integrable boundary values for $p=1$ and
$p=\infty$. Also, the same authors in \cite{KresinMazya2} solved the half-space
analog of this problem for $p=1,\, p=2$ and $p=\infty.$ In fact, they precisely
proved:
\begin{equation}
\label{eq:7}
|\nabla U(x)| \leq \frac{4}{\sqrt{\pi}}
\frac{(n-1)^{\frac{n-1}{2}}}{n^{\frac{n}{2}}}
\frac{\Gamma(\frac{n}{2})}{\Gamma(\frac{n-1}{2})}\frac{1}{x_n}
\sup\limits_{y \in \mathbb{R}_+^n} |U(y)|,
\end{equation}
for a real bounded harmonic function $U$ in the
$n$-dimensional upper-half space $\mathbb{R}_+^n.$
 
\subsection{Partial solutions of generalised Khavinson problem}
Marijan Marković considered the problem in special situation when $x \in
\mathbb{B}^n$ is near the boundary. He confirmed Khavinson conjecture in this
setting and also gave some conclusions and formulas in general. We will appeal
to some of his conclusions later. As first, let us say that he showed that it is
enough to prove the conjecture in special case $x=\rho e_1$ considering only
directions of the form $l=l_{\alpha}=\cos \alpha e_1+ \sin \alpha e_2.$ He
obtained the following formula for the quantity $C(z,\rho)$: 
\begin{equation}
\label{eq:8}
C(z,\rho)=\frac{4\omega_{n-2}}{\omega_{n}}\frac{2^{n-1}}{(1-\rho)(1+\rho)^{n-1}}\frac{1}{\sqrt{1+z^2}}\int_{0}^{1}\frac{\mathcal{P}_{\rho}(zt)+\mathcal{P}_{\rho}(-zt)}{\sqrt{(1-t^2)^{4-n}}} dt, 
\end{equation}
where 
$$\mathcal{P}_{\rho}(z)= \int_{0}^{\frac{z+\sqrt{z^2+1-\alpha^2_{\rho}}}{1-\alpha_{\rho}}} \frac{n-\beta_{\rho}+nzw-\beta_{\rho}w^2}{(1+w^2)^{\frac{n}{2}+1}(1+\kappa^2_{\rho}w^2)^{\frac{n}{2}-1}}w^{n-2} dw $$
with $\alpha_{\rho}=\frac{n-2}{n}\rho$, $\beta_{\rho}=\frac{n-(n-2)\rho}{2}$ and
$\kappa_{\rho}=\frac{1-\rho}{1+\rho}.$ Here, $z=\tan \alpha$ and $\omega_n$ is
the area of $\partial \mathbb{B}^n.$
 
In this circumference, generalized Khavinson conjecture is equivalent with the
fact that $\sup\limits_{z>0} C(z,\rho)=C(0,\rho).$
 
Using the formula \eqref{eq:8}, David Kalaj in \cite{Kalaj} prove the conjecture in the unit
ball in $\mathbb{R}^4.$ But, as it can be seen from the definition of
$\mathcal{P}_{\rho}$, this formula for $C(z,\rho)$ suitable for Marković's
considerations, is not the appropriate one for treating the case when $n$ is
odd. This is the main reason why this problem is considered to be hard
especially for $n \in 2\mathbb{N}+1$.

\subsection{Organisation of the paper and results}
We prove a representation formula which we will prove in the second
section:
\begin{theorem}
For every $0\leq\rho<1$ the following integral representation for $C(\rho e_1, l)$ holds:
\begin{multline}
\label{eq:9}
C(\rho e_1,l)=\frac{n}{1-\rho^2}\frac{\Gamma(\frac{n}{2})}{\Gamma(\frac{1}{2})\Gamma(\frac{n-1}{2})}\\
\times \int_{-1}^1 \lvert \tfrac{n-2}n \rho \cos\alpha - x\rvert
\frac{(1-x^2)^{\frac{n-3}2}}{(1+\rho^2-2\rho x \cos\alpha)^{\frac{n}2-1}} {}_2F_1(\tfrac{n-2}{2},\tfrac{n}4;\tfrac{n-1}2;\tfrac{4\rho^2 \sin^2\alpha (1-x^2)}{(1+\rho^2-2\rho x \cos\alpha)^{2}}) dx.
\end{multline}
\end{theorem}
This formula enabled us to reduce Khavinson problem on finding the
maximum of some function $C(\alpha)$. In fact, since it seems to be very hard to
do that directly, we find the majorant $\tilde{C}(\alpha)$ which satisfies
$C(\alpha)\leq \tilde{C}(\alpha)$ and $C(0)=\tilde{C}(0).$ In order to do this
more effectively, we prove some unexpected integral identities in the third
section, thus obtaining the appropriate majorant in the fourth section. These
identities include hypergeometric functions and seems to have independent
interest. For general information on these functions, see
\cite{AndrewsAskeyRoy}. The last section is devoted to the detailed analysis of
$\tilde{C}(\alpha),$ and the final proof of our main theorem(Conjecture 1.1 in $\mathbb{R}^3$):

\begin{theorem}
If $u$ is a bounded harmonic function in the unit ball $\mathbb{B}^3,$ then we have the following sharp inequality:
 \begin{equation}
 \label{eq:10}
 |\nabla u(x)|\leq \frac{1}{\rho^2}\bigg({\frac{(1+\frac{1}{3}\rho^2)^{\frac{3}{2}}}{1-\rho^2}-1}\bigg) \sup_{|y|<1} |u(y)|, 
 \end{equation}
with $\rho=|x|.$  
\end{theorem}
 
\section{A general representation formula for the sharp constant}
 
In \cite{Markovic1} Marković gives some general observations about the problem
in $\mathbb{B}^n.$ We start from his conclusion that it is enough to prove the
conjecture for $x \in \mathbb{B}^n$ of the form $x=\rho e_1$ and directions
given by $l=l_{\alpha}=e_1 \cos \alpha + e_2 \sin \alpha.$

We start from the formula for the optimal constant in the inequality
\begin{equation}
\label{eq:11}
  \lvert\langle \nabla u(\rho e_1),l\rangle\rvert \leqslant C(\rho e_1,l)\sup_{|y|<1}|u(y)|,
\end{equation}
for $x=\rho e_1$ and the direction $l\in \partial \mathbb{B}^n$ given by:

\begin{equation}
\label{eq:12}
C(\rho e_1,l)=\int_{\partial \mathbb{B}^n} \lvert \langle \nabla P(x,\zeta), l \rangle\rvert 
d\sigma(\zeta), 
\end{equation}

where $P(x,\zeta)=\frac{1-|x|^2}{|x-\zeta|^n}$ is the Poisson kernel for the unit ball $\mathbb{B}^n.$

Mobius transform 
 
$$\zeta=-T_{x}(\eta)=\frac{(1-|x|^2)(\eta-x)-|\eta-x|^2x}{[\eta,x]^2}$$

where 

$$[x,y]=|y||y^{*}-x|,  \quad y^{*}=\frac{y}{|y|^2},  $$

together with some calculations (see \cite{Markovic1} for details) gives us:

\begin{equation} 
\label{eq:13}
C(\rho e_1, l) = \frac{n}{1-\rho^2} \int_{\partial \mathbb{B}^n} \lvert \langle
\eta - \alpha_{\rho} e_1,l \rangle\rvert \lvert \eta-\rho e_1 \rvert^{2-n}
d\sigma(\eta)
\end{equation}
where
$$\alpha_\rho = \frac{n-2}{n}\rho,\quad l=e_1 \cos \alpha + e_2 \sin \alpha, \quad \alpha \in [0,\tfrac{\pi}2].$$
 
So, the problem is, in fact, two-dimensional. We see this, since for fixed $l\in
\partial \mathbb{B}^n$ there exists an orthogonal matrix $A$ such that
\[
Al=e_1,\quad Ae_1 = \cos \alpha e_1 + \sin \alpha e_2,\quad \alpha \in [0,\tfrac{\pi}2].
\]
Then we have:
\begin{align*}
\frac{1-\rho^2}{n} C(\rho e_1,l)
&=\int_{\partial \mathbb{B}^n} \lvert \langle A\eta-\alpha_\rho Ae_1,Al\rangle\rvert \lvert A\eta - \rho Ae_1 \rvert^{2-n} d\sigma(\eta)\\
&= \int_{\partial \mathbb{B}^n} \lvert \langle \xi-\alpha_\rho Ae_1,Al\rangle\rvert \lvert \xi - \rho Ae_1 \rvert^{2-n} d\sigma(\xi)\\
&= \int_{\partial \mathbb{B}^n} \lvert \langle \xi-\alpha_\rho Ae_1,e_1\rangle\rvert \lvert \xi - \rho Ae_1 \rvert^{2-n} d\sigma(\xi)\\
&= \int_{\partial \mathbb{B}^n} \lvert \xi_1-\tfrac{n-2}{n}\rho \cos \alpha\rvert
\lvert\xi - \rho(\cos \alpha e_1 + \sin\alpha e_1)\rvert^{2-n}  d\sigma(\xi).\\
\intertext{Since}
\lvert\xi - \rho(\cos \alpha e_1 &+ \sin\alpha e_2)\rvert^{2-n} \\
&= \langle\xi-\rho\cos\alpha e_1 - \rho \sin \alpha e_2,
\xi-\rho\cos\alpha e_1 - \rho\sin\alpha e_2 \rangle^{1-\frac{n}2} \\
&= (1-2\rho \xi_1 \cos\alpha - 2\rho \xi_2 \sin\alpha + \rho^2)^{1-\frac{n}2}, \\
\end{align*}
we have:
\begin{multline}
\label{eq:14}
\frac{1-\rho^2}{n} C(\rho e_1,l) \\
= \int_{\partial \mathbb{B}^n}
\lvert \xi_1 - \tfrac{n-2}n \rho \cos\alpha\rvert
\lvert 1-2\rho \xi_1 \cos\alpha - 2\rho \xi_2 \sin\alpha + \rho^2 \rvert^{1-\frac{n}2} d\sigma(\xi).  
\end{multline}  

Now, the Khavinson conjecture is equivalent with the fact that, for fixed $\rho
\in [0,1)$ the maximum of the last integral as a function on $\alpha$ is
attained at $\alpha=0.$

To expand the integral in \eqref{eq:14}, let us note that the integrand depends
only on $\xi_1$ and $\xi_2.$ To do the expansion, we prove the following lemma
which gives us the formula for integrals over the sphere $\partial \mathbb{B}^n$ of
functions which depends on $k$ variables. It is a real counterpart of the Lemma
from Rudin's book \cite{Rudin}.

\begin{lemma}\label{prvaLema}
  Let $f$ be a continuous function on the closed ball $\overline{\mathbb{B}^k}$
  which depends on first $k$ variables. If $P$ is projection on $\mathbb{R}^k$,
  we have:
  \begin{equation}
  \label{eq:15}
    \int_{\partial \mathbb{B}^n} (f \circ P)d\sigma =
    \frac{\Gamma(\frac{n}2)}{\Gamma(\frac{k}2+1)\Gamma(\frac{n-k}2)}
    \int_{\mathbb{B}^k} (1-\lvert x \rvert^2)^{\frac{n-k-2}2} f(x)dv_k(x),
  \end{equation}
  where $\sigma$ is normalized area measure and $v_k$ normalised Lebesgue volume
  measure.
\end{lemma}
\begin{proof}
  Let us take $f\in C(\mathbb{B}^k),$ $\operatorname{supp} f\subset r_0 B^k,$
  for some $r_0 <1.$ Then, we define:
  \begin{align*}
  I(r) &= \int_{r \mathbb{B}^n} (f \circ P)dv_n = n \int_0^r t^{n-1}
    \biggl(\int_{\mathbb{S}} (f\circ P)(t\zeta)d\sigma(\zeta)\biggr)dt, \\
  \intertext{from which, after differentiating, we get:}
  I'(1) &= n \int_{\mathbb{S}} (f\circ P)(\zeta)d\sigma(\zeta).\\
  \intertext{On the other side, integrating over the $n-k$ free variables gives us:}
  I(r) &=c_{n,k} \int_{\mathbb{B}^k} (r^2-\lvert x\rvert^2)^{\frac{n-k}2} f(x)
    dv_k(x),\\
  \intertext{so differentiation implies:}
  I'(r) &=  \frac{(n-k)c_{n,k}}2 \cdot 2r \int_{\mathbb{B}^k} (r^2-\lvert x
    \rvert^2)^{\frac{n-k-2}2} f(x)dv_k(x) \\
  \intertext{and hence:}
  I'(1) &= \widetilde{c}_{n,k} \int_{\mathbb{B}^k}
(1-\lvert x \rvert^2)^{\frac{n-k-2}2} f(x)dv_k(x).
  \end{align*}

To find the exact value of the constant that in the statement of our Lemma, let us set $f(x)=1$:
\begin{align*}
1&= \int_{\mathbb{S}} (f \circ P) d\sigma = C_{n,k} \int_{\mathbb{B}^k}
(1-\lvert x \rvert^2)^{\frac{n-k-2}2} dv_k(x)\\
&=k C_{n,k} \int_0^1 t^{k-1} \biggl(\int_{S^k} (1-\lvert t\zeta\rvert^2)^{\frac{n-k-2}2} d\sigma(\zeta)\biggr)dt\\
&=k C_{n,k} \int_0^1 t^{k-1} (1-t^2)^{\frac{n-k}2-1} dt\\
&= \frac{kC_{n,k}}2  \int_0^1 (t^2)^{\frac{k}2-1}
(1-t^2)^{\frac{n-k}2-1} d(t^2)\\
&=\frac{kC_{n,k}}2  \operatorname{B}(\tfrac{k}2,\tfrac{n-k}2)
 = \frac{kC_{n,k}}{2}  \frac{\Gamma(\frac{k}{2})\Gamma(\frac{n-k}2)}{\Gamma(\frac{n}2)},
\end{align*}
therefore:
\begin{equation*}
C_{n,k} = \frac{\Gamma(\frac{n}2)}{\Gamma(\frac{k}{2}+1)\Gamma(\frac{n-k}{2})}.
\qedhere
\end{equation*}
\end{proof}

Applying our Lemma \ref{prvaLema}, we get:
\begin{align*}
&\frac{1-\rho^2}{n} C(\rho e_1,l) =\\
&=\frac{\Gamma(\frac{n}2)}{\pi \Gamma(\frac{n}2-1)\Gamma(2)}
\int_{\mathbb{B}^2} \lvert \tfrac{n-2}{n} \rho \cos \alpha - x\rvert
    \frac{(1-x^2-y^2)^{\frac{n}2-2}}{(1-2\rho x \cos \alpha - 2 \rho y \sin\alpha + \rho^2)^{\frac{n}2-1}} dxdy  \\
&= \frac{n-2}{2\pi}  \int_{-1}^1 \lvert \tfrac{n-2}n \rho \cos\alpha -x \rvert
\biggl(\int_{-\sqrt{1-x^2}}^{\sqrt{1-x^2}} \frac{(1-x^2-y^2)^{\frac{n}2-2}}{(1-2\rho x \cos \alpha - 2 \rho y \sin\alpha + \rho^2)^{\frac{n}2-1}} dy\biggr)dx.
\end{align*}

Calculation of the inner integral will be done in the next lemma. We will invoke
it at the appropriate places in the proof.
\begin{lemma}\label{drugaLemma} 
  For $\rho \in [0,1]$ and $\alpha \in [0,\frac{\pi}{2}] $ there holds the
  following identity:	
\begin{multline}
\label{eq:16}
\int_{-\sqrt{1-x^2}}^{\sqrt{1-x^2}} \frac{(1-x^2-y^2)^{\frac{n}2-2}}{(1-2\rho x
  \cos \alpha - 2 \rho y \sin\alpha + \rho^2)^{\frac{n}2-1}} dy \\
=\operatorname{B}\big(\frac12,
\frac{n}2-1\big) \frac{(1-x^2)^{\frac{n-3}2}}{(1+\rho^2-2\rho x \cos\alpha)^{\frac{n}2-1}} {}_2F_1(\tfrac{n-2}{2},\tfrac{n}4;\tfrac{n-1}2;\tfrac{4\rho^2 \sin^2\alpha (1-x^2)}{(1+\rho^2-2\rho x \cos\alpha)^{2}})
\end{multline}
\end{lemma}
\begin{proof}
We change variable by $y=\sqrt{1-x^2}t$:
\begin{align*}
\int_{-\sqrt{1-x^2}}^{\sqrt{1-x^2}} &\frac{(1-x^2-y^2)^{\frac{n}2-2}}{(1-2\rho x \cos \alpha - 2 \rho y \sin\alpha + \rho^2)^{\frac{n}2-1}} dy \\
&= \int_{-1}^1 \frac{(1-x^2)^{\frac{n}2-2}(1-t^2)^{\frac{n}2-2}\sqrt{1-x^2}}{(1-2\rho x \cos\alpha - 2\rho \sqrt{1-x^2} t \sin \alpha + \rho^2)^{\frac{n}2-1}} dt \\
&= (1-x^2)^{\frac{n-3}2} (1+\rho^2-2\rho x \cos\alpha)^{1-\frac{n}2} \int_{-1}^1 \frac{(1-t^2)^{\frac{n}2-2}}{\Bigl(1-\frac{2\rho t \sqrt{1-x^2} \sin\alpha}{1+\rho^2-2\rho x \cos\alpha}\Bigr)^{\frac{n}2-1}}dt. 
\end{align*}
Using power series expansion
\[ (1-z)^{-\alpha}=\sum_{k=0}^{+\infty} \binom{k+\alpha-1}{k}z^k, \quad
  \text{for } \alpha=\frac{n}{2}-1 \text{ and } z=\frac{2\rho  t \sqrt{1-x^2}
    \sin\alpha}{1+\rho^2-2\rho x \cos\alpha}, \]
we get:
\begin{align*}
\int_{-1}^1 &\frac{(1-t^2)^{\frac{n}{2}-2}}{\Bigl(1-\frac{2\rho t \sqrt{1-x^2} \sin\alpha}{1+\rho^2-2\rho x \cos\alpha}\Bigr)^{\frac{n}{2}-2}}dt \\
&=  \sum_{k=0}^\infty \binom{k+\frac{n}{2}-1-1}{k} \biggl(\frac{2\rho  \sqrt{1-x^2} \sin\alpha}{1+\rho^2-2\rho x \cos\alpha}\biggr)^{\!k} \int_{-1}^1 (1-t^2)^{\frac{n}{2}-2} t^k dt \\
&= 2\sum_{k=0}^\infty \binom{2k+\frac{n}{2}-2}{2k} \frac{4^k \rho^{2k} (1-x^2)^k \sin^{2k}\alpha}{(1+\rho^2-2\rho x \cos\alpha)^{2k}} \int_{0}^1 t^{2k}(1-t^2)^{\frac{n}{2}-2}  dt.
\end{align*}
We easily find that
\begin{align*}
2\int_{0}^1 t^{2k}(1-t^2)^{\frac{n}2-2}  dt &= \int_0^1 u^{k-\frac12}(1-u)^{\frac{n}2-2}du \\
& =\operatorname{B}\big(k+\frac12,\frac{n}2-1\big)
  = \frac{\Gamma(k+\frac12)\Gamma(\frac{n}2-1)}{\Gamma(k+\frac{n-1}2)}.
\end{align*}
Since $\binom{2k+\frac{n}2-2}{2k} = \frac{(\frac{n}2-1)_{2k}}{(2k)!},$ by duplication formula for Gamma function, we have
\begin{align*}
\binom{2k+\frac{n}2-2}{2k} \frac{\Gamma(k+\frac12)\Gamma(\frac{n}2-1)}{\Gamma(k+\frac{n-1}2)}
&= \frac{(\frac{n}2-1)_{2k}\Gamma(k+\frac12)\Gamma(\frac{n}2-1)}
{\Gamma(2k+1)\Gamma(\frac{n-1}2)\bigl(\frac{n-1}2\bigr)_k} \\
&=\operatorname{B}(\tfrac12, \tfrac{n}2-1) \frac{\bigl(\frac{n-2}4\bigr)_k\bigl(\frac{n}4\bigr)_k}{\bigl(\frac{n-1}2\bigr)_k k!}. \qedhere
\end{align*}
\end{proof}

Using now Lemma \ref{drugaLemma} i.e. \eqref{eq:16}, 
we get:
\begin{multline*}
\frac{1-\rho^2}{n}C(\rho e_1,l) = \frac{n-2}{2\pi} \operatorname{B}\big(\frac12,
\frac{n}2-1\big) \times \\
\times \int_{-1}^1 \lvert \tfrac{n-2}n \rho \cos\alpha - x\rvert
\frac{(1-x^2)^{\frac{n-3}2}}{(1+\rho^2-2\rho x \cos\alpha)^{\frac{n}2-1}} {}_2F_1(\tfrac{n-2}{2},\tfrac{n}4;\tfrac{n-1}2;\tfrac{4\rho^2 \sin^2\alpha (1-x^2)}{(1+\rho^2-2\rho x \cos\alpha)^{2}}) dx,
\end{multline*}
i.e. our Theorem 1.

\section{Three important integral identities}

Before we can come to the main estimate, we need three integral identities. We
derive the first two of them from Lemma \ref{prvaLema} using it for some special
choices of the function $f.$ Identities are given by the next lemmata.

\begin{lemma}\label{trecaLema}
  We have the following equality for all $\rho \in [0,1]$ and $\alpha \in
  [0,\frac{\pi}{2}]:$
\begin{multline}
\label{eq:17}
\int_{-1}^1 \frac{(1-x^2)^{\frac{n-3}2}}{(1+\rho^2-2\rho x \cos\alpha)^{\frac{n}2-1}}  {}_2 F_1\bigl(\tfrac{n-2}4, \tfrac{n}4;
\tfrac{n-1}2; \tfrac{4\rho^2 \sin^2\alpha (1-x^2)}{(1+\rho^2-2\rho x \cos\alpha)^2}\bigr) dx\\
=
\int_{-1}^1 (1-x^2)^{\frac{n-3}2} (1-2\rho x + \rho^2)^{1-\frac{n}2} dx.
\end{multline}
\end{lemma}
\begin{proof}
Using Lemma \ref{prvaLema} for $f(x,y)=(1-2\rho x \cos\alpha - 2\rho y \sin\alpha + \rho^2)^{1-\frac{n}2},$ where  $(x,y) \in \mathbb{B}^2$, we get
\begin{align*}
&\int_{\partial \mathbb{B}^n} (1-2\rho \xi_1 \cos\alpha - 2\rho \xi_2 \sin\alpha + \rho^2)^{1-\frac{n}2} d\sigma(\xi)  \\
&= \frac{n-2}{2\pi} \int_{\mathbb{B}^2} (1-x^2-y^2)^{\frac{n}2-2} (1-2\rho x \cos\alpha - 2\rho y \sin\alpha + \rho^2)^{1-\frac{n}2} dxdy  \\
&= \frac{n-2}{2\pi}
\int_{-1}^1 \biggl(\int_{-\sqrt{1-x^2}}^{\sqrt{1-x^2}} (1-x^2-y^2)^{\frac{n}2-2} (1-2\rho x \cos\alpha - 2\rho y \sin\alpha + \rho^2)^{1-\frac{n}2} dy\biggr)dx\\
&=\frac{n-2}{2\pi}\operatorname{B}\big(\frac12, \frac{n}2-1\big)
\int_{-1}^1 \frac{(1-x^2)^{\frac{n-3}2}}{(1+\rho^2-2\rho x \cos\alpha)^{\frac{n}2-1}}  {}_2 F_1\bigl(\tfrac{n-2}4, \tfrac{n}4;
\tfrac{n-1}2; \tfrac{4\rho^2 \sin^2\alpha (1-x^2)}{(1+\rho^2-2\rho x \cos\alpha)^2}\bigr) dx,
\end{align*}
by Lemma \ref{drugaLemma}.

On the other hand, introducing the change of variables, $\zeta_1= \xi_1 \cos\alpha + \xi_2 \sin\alpha,$ $\zeta_2 =\xi_1 \sin\alpha - \xi_2 \cos\alpha$ and $\xi_k=\zeta_k$ for $3\leq k \leq n,$ we get:
\begin{multline*}
\int_{\partial\mathbb{B}^n}(1-2\rho \xi_1 \cos \alpha - 2\rho \xi_2 \sin\alpha + \rho^2)^{1-\frac{n}2} d\sigma(\xi)
= \int_{\partial\mathbb{B}^n}(1-2\rho \zeta_1 + \rho^2)^{1-\frac{n}2} d\sigma(\zeta)\\
= \frac{\Gamma(\frac{n}2)}{2\Gamma(\frac32)\Gamma(\frac{n-1}2)}
\int_{-1}^1 (1-x^2)^{\frac{n-3}2} (1-2\rho x + \rho^2)^{1-\frac{n}2} dx.
\end{multline*}
Comparing these two expressions for $\int_{\partial\mathbb{B}^n}(1-2\rho \xi_1 \cos
\alpha - 2\rho \xi_2 \sin\alpha + \rho^2)^{1-\frac{n}2} d\sigma(\xi)$ we
conclude the proof of \eqref{eq:17}.
\end{proof}

One more necessary identity is given by the following lemma. 
\begin{lemma}\label{cetvrtaLema}
  There holds the following identity for all $\rho \in [0,1]$ and $\alpha \in
  [0,\frac{\pi}{2}]:$
	\begin{multline}
	\label{eq:18}
	\int_{-1}^1 \frac{x^2(1-x^2)^{\frac{n-3}2}}{(1+\rho^2-2\rho x \cos\alpha)^{\frac{n}2-1}}  {}_2 F_1\bigl(\tfrac{n-2}4, \tfrac{n}4;
	\tfrac{n-1}2; \tfrac{4\rho^2 \sin^2\alpha (1-x^2)}{(1+\rho^2-2\rho x
    \cos\alpha)^2}\bigr) dx \\
	= \frac{\sin^2\alpha }{n-1}
   \int_{-1}^{1} (1-x^2)^{\frac{n-1}2} (1-2\rho x + \rho^2)^{1-\frac{n}2} dx \\
    + \cos^2\alpha \int_{-1}^1 x^2(1-x^2)^{\frac{n-3}2} (1-2\rho x + \rho^2)^{1-\frac{n}2} dx.
	\end{multline}
\end{lemma}

\begin{proof}
Similarly as in Lemma \ref{trecaLema}, using Lemma \ref{prvaLema} and Lemma \ref{drugaLemma} for function $f(x,y)=x^2(1-2\rho x \cos\alpha - 2\rho y \sin\alpha + \rho^2)^{1-\frac{n}2},$ we get
\begin{multline*}
\int_{\partial \mathbb{B}^n} \xi_1^2(1-2\rho \xi_1 \cos\alpha - 2\rho \xi_2
\sin\alpha + \rho^2)^{1-\frac{n}2} d\sigma(\xi)  \\
=\frac{n-2}{2\pi}\operatorname{B}\big(\frac12, \frac{n}2-1\big)
\int_{-1}^1 \frac{x^2(1-x^2)^{\frac{n-3}2}}{(1+\rho^2-2\rho x \cos\alpha)^{\frac{n}2-1}}  {}_2 F_1\bigl(\tfrac{n-2}4, \tfrac{n}4;
\tfrac{n-1}2; \tfrac{4\rho^2 \sin^2\alpha (1-x^2)}{(1+\rho^2-2\rho x \cos\alpha)^2}\bigr) dx.
\end{multline*}
On the other hand, introducing the change of variables, $\zeta_1= \xi_1
\cos\alpha + \xi_2\sin\alpha,$ $\zeta_2=\xi_1 \sin\alpha - \xi_2 \cos\alpha$ and
$\zeta_k=\xi_k$ for $3\leq k \leq n,$ we get:
\begin{multline*}
\int_{\partial\mathbb{B}^n}\xi_1^2(1-2\rho \xi_1 \cos \alpha - 2\rho \xi_2 \sin\alpha + \rho^2)^{1-\frac{n}2} d\sigma(\xi)\\
= \int_{\partial\mathbb{B}^n}(\zeta_1 \cos\alpha + \zeta_2 \sin\alpha)^2(1-2\rho \zeta_1 + \rho^2)^{1-\frac{n}2} d\sigma(\zeta)\\
= \cos^2 \alpha \int_{\partial\mathbb{B}^n}\zeta_1^2 (1-2\rho \zeta_1 + \rho^2)^{1-\frac{n}2} d\sigma(\zeta) +
\sin^2 \alpha \int_{\partial\mathbb{B}^n}\zeta_2^2 (1-2\rho \zeta_1 + \rho^2)^{1-\frac{n}2} d\sigma(\zeta).
\end{multline*}
The integral $\int_{\partial\mathbb{B}^n}\zeta_1 \zeta_2(1-2\rho \zeta_1 +
\rho^2)^{1-\frac{n}2} d\sigma(\zeta)$ is equal to zero, since the function under
the integral sign is odd on $\zeta_2.$

These integrals we expand using Lemma \ref{prvaLema} thus obtaining:
\begin{align*}
\int_{\partial\mathbb{B}^n}\zeta_1^2 (1-2&\rho \zeta_1 + \rho^2)^{1-\frac{n}2} d\sigma(\zeta)\\
&=\frac{\Gamma(\frac{n}2)}{2\Gamma(\frac32)\Gamma(\frac{n-1}2)}
\int_{-1}^1 (1-x^2)^{\frac{n-3}2} x^2 (1-2\rho x + \rho^2)^{1-\frac{n}2} dx,\\
\int_{\partial\mathbb{B}^n}\zeta_2^2 (1-2&\rho \zeta_1 + \rho^2)^{1-\frac{n}2} d\sigma(\zeta)\\
&=\frac{\Gamma(\frac{n}2)}{\pi\Gamma(\frac{n}2-1)}
\int_{\mathbb{B}^2}y^2 (1-x^2-y^2)^{\frac{n-4}2} (1-2\rho x + \rho^2)^{1-\frac{n}2} dxdy \\
\intertext{Second integral we calculate integrating first over $y$-variable:}
\int_{\mathbb{B}^2}y^2 (1-&x^2-y^2)^{\frac{n-4}2} (1-2\rho x + \rho^2)^{1-\frac{n}2} dxdy \\
&=\int_{-1}^{1}  (1-2\rho x + \rho^2)^{1-\frac{n}2} \bigg(\int_{-\sqrt{1-x^2}}^{\sqrt{1-x^2}} y^2(1-x^2-y^2)^{\frac{n-4}{2}} dy  \bigg) dx \\
&=\operatorname{B}\big(\frac32, \frac{n}2-1\big) \int_{-1}^{1} (1-x^2)^{\frac{n-1}{2}}  (1-2\rho x + \rho^2)^{1-\frac{n}2} dx.
\end{align*}
The procedure is the same as in the some of previous calculations. Finally,
equalizing two expressions for the same integral we get the identity \eqref{eq:18}.
\end{proof}

The last lemma, we mention here has already been proved by Marković and its
half-space counterpart by Maz'ya and Kresin, but we will also give an easy and
quick proof.

\begin{lemma}\label{petaLema}
  For all $\rho \in [0,1)$ and $\alpha \in [0,\frac{\pi}{2}],$ we have the next
  identity:
\begin{equation}
\label{eq:19}
\int_{-1}^1 \frac{\big(\frac{n-2}{n}\rho\cos\alpha-x\big)\big(1-x^2\big)^{\frac{n-3}2}}{(1+\rho^2-2\rho x \cos\alpha)^{\frac{n}2-1}}  {}_2 F_1\bigl(\tfrac{n-2}4, \tfrac{n}4;
\tfrac{n-1}2; \tfrac{4\rho^2 \sin^2\alpha (1-x^2)}{(1+\rho^2-2\rho x \cos\alpha)^2}\bigr) dx=0.
\end{equation}	
\end{lemma}
\begin{proof}
Denote that the integrand is similar to the representation of $C(\rho e_1,v)$	with the one, but crucial difference---we do not have absolute brackets around the $\frac{n-2}{n}\rho\cos\alpha-x.$ So, since all transformations that we apply to obtain our integral expression for $C(\rho e_1,v)$	save equality without these brackets, we have that 
$$\int_{-1}^1 \frac{\big(\frac{n-2}{n}\rho\cos\alpha-x\big)\big(1-x^2\big)^{\frac{n-3}2}}{(1+\rho^2-2\rho x \cos\alpha)^{\frac{n}2-1}}  {}_2 F_1\bigl(\tfrac{n-2}4, \tfrac{n}4;
\tfrac{n-1}2; \tfrac{4\rho^2 \sin^2\alpha (1-x^2)}{(1+\rho^2-2\rho x \cos\alpha)^2}\bigr) dx$$
is, in fact, by \eqref{eq:12}, equal to $\langle \nabla P(z,\zeta)f(\zeta), (1,1,\dots,1)\rangle,$ on $\partial\mathbb{B}^n$ for function $f(\zeta)=1$ on $\partial\mathbb{B}^n$ and $z \in \mathbb{B}^n$. But this is equal to is, in fact, equal to $\langle \nabla f(z), (1,1,\dots,1)\rangle,$ on $\mathbb{B}^n$ and since $f(z)=1,$ by the uniqueness of harmonic extension, the proof follows.	
\end{proof}

\section{Construction of the majorant}

As we have said in the introduction, the crux of the proof is construction of
the majorant with the maximum in $\alpha=0$ and the same value in $\alpha=0$ as
$C(\rho e_1,l).$

Starting from the general representation formula, we split the hypergeometric
function into two parts:
\begin{multline*}
\int_{-1}^1 \frac{|\frac{n-2}{n}\rho\cos\alpha-x|(1-x^2)^{\frac{n-3}2}}{(1+\rho^2-2\rho x \cos\alpha)^{\frac{n}2-1}}  {}_2 F_1\bigl(\tfrac{n-2}4, \tfrac{n}4;
\tfrac{n-1}2; \tfrac{4\rho^2 \sin^2\alpha (1-x^2)}{(1+\rho^2-2\rho x
  \cos\alpha)^2}\bigr) dx \\
=\int_{-1}^1
\frac{|\frac{n-2}{n}\rho\cos\alpha-x|(1-x^2)^{\frac{n-3}2}}{(1+\rho^2-2\rho x
  \cos\alpha)^{\frac{n}2-1}} dx \\
+\int_{-1}^1 \frac{|\frac{n-2}{n}\rho\cos\alpha-x|(1-x^2)^{\frac{n-3}2}}{(1+\rho^2-2\rho x \cos\alpha)^{\frac{n}2-1}}  \big({}_2 F_1\bigl(\tfrac{n-2}4, \tfrac{n}4;
\tfrac{n-1}2; \tfrac{4\rho^2 \sin^2\alpha (1-x^2)}{(1+\rho^2-2\rho x \cos\alpha)^2}\bigr)-1\big) dx.
\end{multline*}

Now, lemmas from the previous section suggest that we can evaluate the last
integral with $1$ or $\big(\frac{n-2}{n}\rho\cos\alpha-x\big)^2$ in place of
$|\frac{n-2}{n}\rho\cos\alpha-x|,$ so we estimate it by Cauchy-Schwarz
inequality:
\begin{multline*}
\int_{-1}^1 \frac{|\frac{n-2}{n}\rho\cos\alpha-x|(1-x^2)^{\frac{n-3}2}}{(1+\rho^2-2\rho x \cos\alpha)^{\frac{n}2-1}}  \big({}_2 F_1\bigl(\tfrac{n-2}4, \tfrac{n}4;
\tfrac{n-1}2; \tfrac{4\rho^2 \sin^2\alpha (1-x^2)}{(1+\rho^2-2\rho x \cos\alpha)^2}\bigr)-1\big) dx \\
\leq \bigg(\int_{-1}^1 \frac{(1-x^2)^{\frac{n-3}2}}{(1+\rho^2-2\rho x \cos\alpha)^{\frac{n}2-1}}  \big({}_2 F_1\bigl(\tfrac{n-2}4, \tfrac{n}4;
\tfrac{n-1}2; \tfrac{4\rho^2 \sin^2\alpha (1-x^2)}{(1+\rho^2-2\rho x
  \cos\alpha)^2}\bigr)-1\big) dx \bigg)^{\frac{1}{2}} \times \\
\times  \bigg(\int_{-1}^1 \frac{\big(\frac{n-2}{n}\rho\cos\alpha-x\big)^2(1-x^2)^{\frac{n-3}2}}{(1+\rho^2-2\rho x \cos\alpha)^{\frac{n}2-1}}  \big({}_2 F_1\bigl(\tfrac{n-2}4, \tfrac{n}4;
\tfrac{n-1}2; \tfrac{4\rho^2 \sin^2\alpha (1-x^2)}{(1+\rho^2-2\rho x \cos\alpha)^2}\bigr)-1\big) dx \bigg)^{\frac{1}{2}}.
\end{multline*}

Let us denote:
\begin{equation}
\label{eq:20}
S(\alpha) =\int_{-1}^1 \frac{|\frac{n-2}{n}\rho\cos\alpha-x|(1-x^2)^{\frac{n-3}2}}{(1+\rho^2-2\rho x \cos\alpha)^{\frac{n}2-1}} dx,
\end{equation}
\begin{equation}
\label{eq:21}
S_1(\alpha)=\int_{-1}^1 \frac{(1-x^2)^{\frac{n-3}2}}{(1+\rho^2-2\rho x \cos\alpha)^{\frac{n}2-1}}  \big({}_2 F_1\bigl(\tfrac{n-2}4, \tfrac{n}4;
\tfrac{n-1}2; \tfrac{4\rho^2 \sin^2\alpha (1-x^2)}{(1+\rho^2-2\rho x \cos\alpha)^2}\bigr)-1\big) dx 
\end{equation}
and
\begin{equation}
\label{eq:22}
S_2(\alpha) = \int_{-1}^1 \frac{\big(\frac{n-2}{n}\rho\cos\alpha-x\big)^2(1-x^2)^{\frac{n-3}2}}{(1+\rho^2-2\rho x \cos\alpha)^{\frac{n}2-1}}  \big({}_2 F_1\bigl(\tfrac{n-2}4, \tfrac{n}4;
              \tfrac{n-1}2; \tfrac{4\rho^2 \sin^2\alpha (1-x^2)}{(1+\rho^2-2\rho x \cos\alpha)^2}\bigr)-1\big) dx.
\end{equation}            
Majorant for $C(\rho e_1,l_{\alpha})$ which we have searched for is $\tilde{C}=\frac{n}{1-\rho^2}\frac{n-2}{2\pi}\operatorname{B}\big(\frac12,
\frac{n}2-1\big)\big(S+\sqrt{S_1S_2}\big).$ Denote that $C(0)=\tilde{C}(0),$ as it is needed.

\section{Proof of the Theorem 2}

In this section we will find the explicit formulas for the functions
$S(\alpha),\, S_1(\alpha)$ and $S_2(\alpha)$ when $n=3.$ We see that:

\begin{align*}
\int_{-1}^1 &x^2(1-2\rho x + \rho^2)^{-\frac12} dx \\
&= -\frac{1}{15 \rho^3} \bigl[\sqrt{1+\rho^2-2\rho x}
(2 \rho^4 + 2\rho^3x + \rho^2(3x^2+4)+2+2\rho x)\bigr]\bigl.\bigr|_{-1}^1\\
            &=\frac{1}{15 \rho^3} \bigl(\sqrt{1+\rho^2+2\rho}(2\rho^4-2\rho^3+7\rho^2+2-2\rho) \\
  &\hphantom{=}- \sqrt{1+\rho^2-2\rho}(2\rho^4+2\rho^3+7\rho^2+2+2\rho)\bigr)\\
&= \frac{4\rho^5+10\rho^3}{15\rho^3} = \frac{4}{15} \rho^2+\frac23.
\intertext{and}
\int_{-1}^1 &\frac{1-x^2}{\sqrt{1+\rho^2-2\rho x}}dx \\
&= \frac{1}{15\rho^3}
\bigl[\sqrt{1+\rho^2-2\rho x} (2\rho^4+2\rho^3 x + \rho^2(3x^2-11)+2\rho x+2)\bigr]\bigl.\bigr|_{-1}^1 \\
&= \tfrac43 - \tfrac{4}{15}\rho^2
\end{align*}
thus, by Lemma \ref{cetvrtaLema} obtaining that
\begin{multline*}
 \int_{-1}^1 \frac{x^2}{\sqrt{1+\rho^2-2\rho x \cos\alpha}}
{}_2F_1\bigl(\tfrac14,\tfrac34,1; \tfrac{4\rho^2 \sin^2\alpha (1-x^2)}{(1+\rho^2-2\rho x \cos\alpha)^2}\bigr) dx\\
= (\tfrac{4}{15}\rho^2 + \tfrac23)\cos^2 \alpha + (\tfrac23-\tfrac{2}{15}\rho^2)\sin^2\alpha
\end{multline*}
Also, from Lemma \ref{trecaLema}, we get:
\begin{multline*}
\int_{-1}^1 \frac{{}_2F_1(\tfrac{1}4,\tfrac{3}4,1; \tfrac{4\rho^2 (1-x^2)\sin^2\alpha}{(1+\rho^2-2\rho x \cos\alpha)^2})}{\sqrt{1+\rho^2-2\rho x \cos\alpha}} dx =
\int_{-1}^1 (1-2\rho x + \rho^2)^{-\frac12} dx\\
=-\frac{1}{\rho} \sqrt{1+\rho^2-2\rho x}\,\Bigl.\Bigr|_{-1}^1 = \frac{(1+\rho)-(1-\rho)}{\rho} = 2
\end{multline*}

We need also the following two integrals:
\begin{align*}
&\int_{-1}^1 \frac{(\frac13 \rho\cos\alpha - x)^2}{\sqrt{1+\rho^2-2\rho x\cos\alpha}}dx\\
&= \frac{\sqrt{1+\rho^2-2\rho x\cos\alpha}}{360\rho^3 \cos^3\alpha} \Bigl(-5\rho^4\cos(4\alpha)-72\rho^2 x^2\cos^2\alpha\\
&\phantom{=}\,+8\rho x\cos\alpha(5\rho^2 \cos(2\alpha)-\rho^2-6)+20\rho^2(2+\rho^2)\cos(2\alpha)-23\rho^4-56\rho^2-48 \Bigr) \Bigr|_{-1}^1\\
&= \frac{1}{360\rho^3 \cos^3\alpha} \Bigl((-40\rho^4\cos^4\alpha+8\rho^2(10\rho^2+1)\cos^2\alpha-48(1+\rho^2)^2\Bigr)\\
&\phantom{=}\,\times\bigg(\sqrt{1+\rho^2-2\rho \cos\alpha}-\sqrt{1+\rho^2+2\rho \cos\alpha}\bigg)\\
&\phantom{=}\,+\frac{1}{360\rho^3 \cos^3\alpha}\bigg(80\rho^3\cos^3\alpha-48\rho^3\cos\alpha-48\rho\cos\alpha\bigg)\\
&\phantom{=}\,\times\Bigl( \sqrt{1+\rho^2+2\rho \cos\alpha}+\sqrt{1+\rho^2-2\rho \cos\alpha}\Bigr)
\end{align*}
and 
\begin{align*}
\int_{-1}^1 \frac{dx}{\sqrt{1+\rho^2-2\rho x \cos\alpha}} &=
-\frac{\sqrt{1+\rho^2-2\rho x \cos\alpha}}{\rho\cos\alpha}\biggl.\biggr|_{-1}^1\\
&= \frac{\sqrt{1+\rho^2+2\rho  \cos\alpha}-\sqrt{1+\rho^2-2\rho  \cos\alpha}}{\rho\cos\alpha}.
\end{align*}

Now, appealing to Lemmas \ref{trecaLema},\ref{cetvrtaLema} and \ref{petaLema}, we find $S_1(\alpha)$ and $S_2(\alpha):$
\begin{align*}
S_1(\alpha)&=2-\frac{\sqrt{1+\rho^2+2\rho  \cos\alpha}-\sqrt{1+\rho^2-2\rho  \cos\alpha}}{\rho\cos\alpha}, \\
S_2(\alpha)&=(\tfrac23+\tfrac{4}{15}\rho^2)\cos^2\alpha+(\tfrac23-\tfrac{2}{15}\rho^2)\sin^2\alpha
-\tfrac19\rho^2\cos^2\alpha\cdot 2\\
&\phantom{=}\,+\Bigl(\frac{\rho\cos\alpha}{9}-\frac{10\rho^2+1}{45\rho\cos\alpha}+\frac{2(1+\rho^2)^2}{15\rho^3\cos^3\alpha}\Bigr)\\
&\phantom{=}\,\times  \bigg(\sqrt{1+\rho^2-2\rho \cos\alpha}-\sqrt{1+\rho^2+2\rho \cos\alpha}\bigg)\\
&\phantom{=}\,+\Bigl(\frac{2(1+\rho^2)}{15\rho^2\cos^2\alpha}-\frac{2}{9}\Bigr)\Bigl( \sqrt{1+\rho^2+2\rho \cos\alpha}+\sqrt{1+\rho^2-2\rho \cos\alpha}\Bigr)\\
&=(\tfrac23+\tfrac{2}{45}\rho^2)\cos^2\alpha+(\tfrac23-\tfrac{2}{15}\rho^2)\sin^2\alpha\\
&\phantom{=}\,+\Bigl(\frac{\rho\cos\alpha}{9}-\frac{10\rho^2+1}{45\rho\cos\alpha}+\frac{2(1+\rho^2)^2}{15\rho^3\cos^3\alpha}\Bigr)\\
&\phantom{=}\,\times  \Bigl(\sqrt{1+\rho^2-2\rho \cos\alpha}-\sqrt{1+\rho^2+2\rho \cos\alpha}\Bigr)\\
&\phantom{=}\,+\Bigl(\frac{2(1+\rho^2)}{15\rho^2\cos^2\alpha}-\frac{2}{9}\Bigr)\Bigl( \sqrt{1+\rho^2+2\rho \cos\alpha}+\sqrt{1+\rho^2-2\rho \cos\alpha}\Bigr)\\
\end{align*}
Also, using 
\begin{multline*}
  \int \frac{\frac13 \rho\cos\alpha - x}{\sqrt{1+\rho^2-2\rho x\cos\alpha}}dx \\
  = \frac{\sqrt{1+\rho^2-2\rho x
    \cos\alpha}}{3\rho^2\cos^2\alpha}(1+\rho^2-\rho^2\cos^2\alpha +\rho x
\cos\alpha)
\end{multline*}
we find
\begin{align*}
S(\alpha) &= \int_{-1}^1 \frac{\lvert\frac13 \rho\cos\alpha - x\rvert}{\sqrt{1+\rho^2-2\rho x\cos\alpha}}dx \\
&=\frac{1}{3\rho^2\cos^2\alpha}\Bigl(2(1+\rho^2-\frac23\rho^2\cos^2\alpha)^{\frac32}\\
&\qquad\qquad\qquad-(1+\rho^2-\rho^2\cos^2\alpha+\rho\cos\alpha)\sqrt{1+\rho^2-2\rho\cos\alpha}\\
&\qquad\qquad\qquad-(1+\rho^2-\rho^2\cos^2\alpha-\rho\cos\alpha)\sqrt{1+\rho^2+2\rho\cos\alpha}\Bigr)
\end{align*}

To find a majorant, with which we can handle more effectively, we proceed in the
following manner. We estimate $\sqrt{S_1S_2}$ from the above with
$\sqrt{S_1S_2} \leq t S_1+ \frac{1}{4t} S_2,$ by the arithmetic-geometric mean
inequality with $t=\frac{1}{3}$ and therefore get the majorant
$S+\frac{1}{3}S_1+\frac{3}{4}S_2.$ Denote that this majorant has the same value for $\alpha=0$ as $\frac{3}{2(1-\rho^2)}C(\rho e_1,l),$ with  $l=l_{\alpha}.$

We easily calculate 
\begin{multline}
\label{eq:23}
S+\frac{1}{3}S_1+\frac{3}{4}S_2=\frac{2(1+\rho^2-\frac{2}{3}\rho^2
  \cos^2\alpha)^{\frac{3}{2}}}{3
  \rho^2\cos^2\alpha}+\frac{7}{6}-\frac{\rho^2}{10}+\frac{2\rho^2\cos^2\alpha}{15} \\
+\bigg(\frac{1}{6}-\frac{7(1+\rho^2)}{30\rho^2\cos^2\alpha}\bigg)\bigg(\sqrt{1+\rho^2-2\rho\cos\alpha}+\sqrt{1+\rho^2+2\rho\cos\alpha}\bigg) \\
+\bigg(\frac{\rho\cos\alpha}{12}-\frac{10\rho^2+1}{60\rho\cos\alpha}+\frac{(1+\rho^2)^2}{10\rho^3\cos^3\alpha}\bigg)\bigg(\sqrt{1+\rho^2-2\rho\cos\alpha}-\sqrt{1+\rho^2+2\rho\cos\alpha}\bigg).
\end{multline}
(This formula we use for $\alpha<\frac{\pi}{2},$ while for $\alpha=\frac{\pi}{2}$ we can calculate this majorant from the integral expressions. Also, we expect certain cancellations to achieve the function bounded for $\alpha=\frac{\pi}{2}.$)

We can expand square roots using binomial series as:
\begin{multline*}
\sqrt{1+\rho^2-2\rho\cos\alpha}+\sqrt{1+\rho^2+2\rho\cos\alpha} \\
=\sqrt{1+\rho^2}\bigg( \sum_{k=0}^{+\infty} \binom{\frac{1}{2}}{k}(-1)^k\bigg(\frac{2\rho}{1+\rho^2}\bigg)^k\cos^k\alpha+ \sum_{k=0}^{+\infty} \binom{\frac{1}{2}}{k}\bigg(\frac{2\rho}{1+\rho^2}\bigg)^k\cos^k\alpha\bigg) \\
=2\sqrt{1+\rho^2}\sum_{k=0}^{+\infty}\binom{\frac{1}{2}}{2k}\bigg(\frac{2\rho}{1+\rho^2}\bigg)^{2k}\cos^{2k}\alpha.
\end{multline*}
Similarly:
\begin{multline*}
  \sqrt{1+\rho^2-2\rho\cos\alpha}-\sqrt{1+\rho^2+2\rho\cos\alpha} \\
  =-2\sqrt{1+\rho^2}\sum_{k=0}^{+\infty}\binom{\frac{1}{2}}{2k+1}\bigg(\frac{2\rho}{1+\rho^2}\bigg)^{2k+1}\cos^{2k+1}\alpha
\end{multline*}
and 
\begin{multline*}
\begin{split}
\frac{2(1+\rho^2-\frac{2}{3}\rho^2 \cos^2\alpha)^{\frac{3}{2}}}{3 \rho^2\cos^2\alpha} &= \frac{2(1+\rho^2)^{\frac{3}{2}}}{3\rho^2\cos^2\alpha}\bigg(1-\frac{2\rho^2\cos^2\alpha}{3(1+\rho^2)}\bigg)^{\frac{3}{2}} \\
&=\frac{2(1+\rho^2)^{\frac{3}{2}}}{3\rho^2\cos^2\alpha} \sum_{k=0}^{+\infty}
\binom{\frac{3}{2}}{k} (-1)^k \bigg(\frac{2\rho^2}{3(1+\rho^2)}\bigg)^k
\cos^{2k} \alpha \end{split} \\
=\frac{2(1+\rho^2)^{\frac{3}{2}}}{3\rho^2\cos^2\alpha}-\frac{2\sqrt{1+\rho^2}}{3}+\frac{8\rho^2}{27\sqrt{1+\rho^2}}\sum_{k=0}^{+\infty} \binom{\frac{3}{2}}{k+2} (-1)^k \bigg(\frac{2\rho^2}{3(1+\rho^2)}\bigg)^k \cos^{2k+2} \alpha.
\end{multline*}

These expansions give us:
\begin{equation}
\label{eq:24}
 S+\frac{1}{3}S_1+\frac{3}{4}S_2=\frac{7}{6}-\frac{\rho^2}{10}-\frac{1}{6\sqrt{1+\rho^2}}+\sum_{k=1}^{+\infty} a_k(\rho) \cos^{2k}\alpha,
\end{equation}
where
\begin{align*}
a_k(\rho) = \quad &\frac{4\sqrt{1+\rho^2}}{9} \binom{\frac{3}{2}}{k+1} (-1)^{k+1} \bigg(\frac{2\rho^2}{3(1+\rho^2)}\bigg)^{k}+\frac{\sqrt{1+\rho^2}}{3} \binom{\frac{1}{2}}{2k}\bigg(\frac{2\rho}{1+\rho^2}\bigg)^{2k} \\
 {}-{} &\frac{28}{15\sqrt{1+\rho^2}}\binom{\frac{1}{2}}{2k+2}\bigg(\frac{2\rho}{1+\rho^2}\bigg)^{2k}+\frac{10\rho^2+1}{15\sqrt{1+\rho^2}}\binom{\frac{1}{2}}{2k+1}\bigg(\frac{2\rho}{1+\rho^2}\bigg)^{2k} \\
{}-{} &\frac{\rho\sqrt{1+\rho^2}}{6}\binom{\frac{1}{2}}{2k-1}\bigg(\frac{2\rho}{1+\rho^2}\bigg)^{2k-1}-\frac{8}{5\sqrt{1+\rho^2}}
\binom{\frac{1}{2}}{2k+3}\bigg(\frac{2\rho}{1+\rho^2}\bigg)^{2k},
\end{align*}
for $k\geq 1.$ 
Specially, we have
$$a_1(\rho)=\rho^2\bigg(\frac{2}{15}-\frac{1}{18\sqrt{1+\rho^2}}-\frac{1}{30(1+\rho^2)^{\frac{5}{2}}}\bigg).$$
Formulas
\[ \binom{\frac{1}{2}}{2k}=-\frac{(4k-3)!!}{2^{2k}(2k)!}, \
  \binom{\frac{1}{2}}{2k+1}=\frac{(4k-1)!!}{2^{2k+1}(2k+1)!}, \
  (-1)^{k+1}\binom{\frac{3}{2}}{k+1}=\frac{3(2k-3)!!}{2^{k+1}(k+1)!}, \]
for $k\geq2$, gives us
\begin{align*}
a_k(\rho) &=\frac{2\sqrt{1+\rho^2}(2k-3)!!}{3(k+1)!}\Bigl(\frac{\rho^2}{3(1+\rho^2)}\Bigr)^{k}+ \frac{(4k-5)!!}{2^{2k}(2k+3)!\sqrt{1+\rho^2}}\bigg(\frac{2\rho}{1+\rho^2}\bigg)^{2k}\times \\
  {} \times {} &\bigg( -\frac{1+\rho^2}{3}(4k-3)(2k+1)(2k+2)(2k+3)+\frac{7}{15}(2k+3)(4k-3)(4k-1)(4k+1) \\
    &\hphantom{\bigg(} +\frac{1+\rho^2}{3}(2k+2)(2k+3)(4k-3)(4k-1)-\frac{3}{10}(4k-3)(4k-1)(2k+2)(2k+3) \\
  &\hphantom{\bigg(}-\frac{(1+\rho^2)^2}{6}(2k)(2k+1)(2k+2)(2k+3)-\frac{1}{5}(4k-3)(4k-1)(4k+1)(4k+3)\bigg), \\
\intertext{or, after some routine calculations:}
a_k(\rho) &=\frac{2\sqrt{1+\rho^2}(2k-3)!!}{3(k+1)!}\Bigl(\frac{\rho^2}{3(1+\rho^2)}\Bigr)^{k}\\ &+\frac{(4k-5)!!}{(2k+3)!\sqrt{1+\rho^2}}\bigg(\frac{\rho}{1+\rho^2}\bigg)^{2k} \bigg(-(1+\rho^2)^2\bigg(\frac{8}{3}k^4+8k^3+\frac{22}{3}k^2+2k\bigg) \\
{}+{}&
\bigg(1+\rho^2\bigg)\bigg(\frac{32}{3}k^4+8k^3-\frac{68}{3}k^2-8k+12\bigg)-\frac{32}{3}k^4+16k^3-\frac{70}{3}k^2+17k-3\bigg).
\end{align*}

To complete the proof of our main theorem, we need the following two lemmas.

\begin{lemma}\label{sestaLema}
Coefficients $a_k(\rho)$ are negative for all $k\geq 2$ and $0<\rho<1.$
\end{lemma}
\begin{proof}
 We easily see that 
 \begin{multline*}
 a_{k}\rho^{-2k}3^k(1+\rho^2)^{k-\frac{1}{2}}=\frac{(4k-5)!!}{(2k+3)!}\frac{3^k}{(1+\rho^2)^{k+1}}\bigg(-P(k)(1+\rho^2)^2+Q(k)(1+\rho^2)+R(k)\bigg)\\
 +\frac{2}{3}\frac{(2k-3)!!}{(k+1)!},
  \end{multline*}
 where $P(k)=\frac{8}{3}k^4+8k^3+\frac{22}{3}k^2+2k,$ $Q(k)=\frac{32}{3}k^4+8k^3-\frac{68}{3}k^2-8k+12$ and $R(k)=-\frac{32}{3}k^4+16k^3-\frac{70}{3}k^2+17k-3.$
 
 Therefore, the sign of the $a_k$ is determined by the sign of the expression on the right hand side of the previous equation. Let us consider the function:
  $$\Phi(y)=\frac{3^k(4k-5)!!}{(2k+3)!}\bigg[-P(k)y^{-k+1}+Q(k)y^{-k}+R(k)y^{-k-1}\bigg] +\frac{2}{3}\frac{(2k-3)!!}{(k+1)!}.$$
 First, we will prove that it is monotone increasing on $y$. Namely, its first derivative is equal to:
 $$\Phi'(y)=\frac{3^k(4k-5)!!}{(2k+3)!}y^{-k-2} \big[(k-1)P(k)y^2-kQ(k)y-(k+1)R(k)\big], $$
 and the quadratic polynomial $(k-1)P(k)y^2-kQ(k)y-(k+1)R(k)$ is positive for all $k \geq 2$ and $1\leq y \leq 2,$ since its discriminant $D=k^2Q^2(k)+4(k^2-1)P(k)R(k)=-\frac{3200}{9}k^8-544k^7+\frac{5824}{9}k^6+952k^5-\frac{3488}{9}k^4-432k^3+96k^2+24k$ is negative for $k \geq 2.$

 Hence, $\Phi(y)\leq \Phi(2),$ which gives us:
 
 $$\frac{(4k-5)!!}{(2k+3)!}\frac{3^k}{(1+\rho^2)^{k+1}}[-P(k)(1+\rho^2)^2+Q(k)(1+\rho^2)+R(k)]+\frac{2}{3}\frac{(2k-3)!!}{(k+1)!}$$
 $$\leq \frac{(4k-5)!!}{2(2k+3)!}\bigg(\frac{3}{2}\bigg)^k[-4P(k)+2Q(k)+R(k)]+\frac{2}{3}\frac{(2k-3)!!}{(k+1)!},$$
 
 and, taking into account that $-4P(k)+2Q(k)+R(k)=-98k^2-7k+21,$ to finish the proof of the lemma, we must prove
 
 $$ l_k=\frac{(4k-5)!!(k+1)!}{(2k-3)!!(2k+3)!}\bigg(\frac{3}{2}\bigg)^{k+1}> \frac{2}{98k^2+7k-21}=d_k. $$
 
But, since $l_{k+1}=\frac{3(4k-1)(4k-3)(k+2)}{2(2k+5)(2k+4)(2k-1)}l_k$ and $\frac{3(4k-1)(4k-3)(k+2)}{2(2k+5)(2k+4)(2k-1)}>1$ for all $k \geq 2,$ we see that $l_k$ is increasing, while $d_k$ is decreasing. It is easy to check that $l_2>d_2$ and therefore the proof of this lemma is complete.

\end{proof}

Let us consider, like in \eqref{eq:24}, the function 

\begin{equation}
\label{eq:25}
 T(t)=\frac{7}{6}-\frac{\rho^2}{10}-\frac{1}{6\sqrt{1+\rho^2}}+\sum_{k=1}^{+\infty} a_k(\rho)t^k.
\end{equation}

The next lemma considers behavior od this function near the point $t=1.$

\begin{lemma}\label{sedmaLema}
$T_{-}'(1)\geq 0.$
\end{lemma}

\begin{proof}
Let us first denote that, in fact, $(S+\frac{1}{3}S_1+\frac{3}{4}S_2)(\alpha)=T(t),$ for $t=\cos^2\alpha.$ Hence, we have:

$$T_{-}'(1)=\lim\limits_{\alpha \rightarrow 0+} \frac{(S+\frac{1}{3}S_1+\frac{3}{4}S_2)'(\alpha)}{-2\sin\alpha \cos\alpha}.$$

Differentiating \eqref{eq:23}, we get:

$$(S+\frac{1}{3}S_1+\frac{3}{4}S_2)'(\alpha)=\frac{4\sin\alpha\sqrt{1+\rho^2-\frac{2}{3}\rho^2\cos^2\alpha}}{3\rho^2\cos^3\alpha}\big(1+\rho^2+\frac{1}{3}\rho^2\cos^2\alpha\big)$$
$$-\frac{4}{15}\rho^2\cos\alpha\sin\alpha+\bigg(\frac{1}{6}-\frac{7(1+\rho^2)}{30\rho^2\cos^2\alpha}\bigg)\bigg(\frac{-\rho\sin\alpha}{\sqrt{1+\rho^2+2\rho\cos\alpha}}+\frac{\rho\sin\alpha}{\sqrt{1+\rho^2-2\rho\cos\alpha}}\bigg)$$
$$-\frac{7(1+\rho^2)\sin\alpha}{15\rho^2\cos^3\alpha}\bigg(\sqrt{1+\rho^2+2\rho\cos\alpha}+\sqrt{1+\rho^2-2\rho\cos\alpha}\bigg)$$
$$+\bigg(\frac{\rho\cos\alpha}{12}-\frac{10\rho^2+1}{60\rho\cos\alpha}+\frac{(1+\rho^2)^2}{10\rho^3\cos^3\alpha}\bigg)\bigg(\frac{\rho\sin\alpha}{\sqrt{1+\rho^2-2\rho\cos\alpha}}+\frac{\rho\sin\alpha}{\sqrt{1+\rho^2+2\rho\cos\alpha}}\bigg)$$
$$+\bigg(-\frac{\rho\sin\alpha}{12}-\frac{(10\rho^2+1)\sin\alpha}{60\rho\cos^2\alpha}+\frac{3(1+\rho^2)^2\sin\alpha}{10\rho^3\cos^4\alpha}\bigg)\times$$
$$\times\bigg(\sqrt{1+\rho^2-2\rho\cos\alpha}-\sqrt{1+\rho^2+2\rho\cos\alpha}\bigg)$$

Therefore:

$$\lim\limits_{\alpha \rightarrow 0+} \frac{(S+\frac{1}{3}S_1+\frac{3}{4}S_2)'(\alpha)}{-2\sin\alpha \cos\alpha}=\frac{1}{30\rho^2}\bigg[-40(1+\frac{4}{3}\rho^2)\sqrt{1+\frac{1}{3}\rho^2}+11\rho^4+60\rho^2+40\bigg]\geq0,$$
since this is equivalent with:

$$(11\rho^4+60\rho^2+40)^2-1600(1+\frac{4}{3}\rho^2)^2(1+\frac{1}{3}\rho^2)=121\rho^8+\frac{10040}{27}\rho^6+\frac{640}{3}\rho^4\geq 0.$$

Now the lemma follows.

\end{proof}

We can now finish the proof of our main theorem. In fact, we have to conclude that $(S+\frac{1}{3}S_1+\frac{3}{4}S_2)$ has its maximum in $\alpha=0,$ or, what is the same, $T(t),$ defined by \eqref{eq:25}, attains its maximum  in $t=1.$ Differentiating $T(t)$ two times, we get $T''(t)=\sum_{k=2}^{+\infty}k(k-1)a_kt^{k-2},$ and  Lemma \ref{sestaLema} implies that $T''(t)\leq 0.$ Hence, $T'(t)$ is decreasing and $T'(t)\geq T_{-}'(1)$, while Lemma \ref{sedmaLema} gives us $T_{-}'(1)\geq 0.$ Therefore, $T'(t)\geq 0,$ $T$ increases and has the maximum in $t=1;$ consequently $C(\rho e_1, l_{\alpha}) \leq C(\rho e_1,e_1).$

\textbf{Acknowledgements}. The author wishes to express his gratitude to Miroslav Pavlovi\'c, Milan Lazarevi\'c and Nikola Milinkovi\'c for their valuable suggestions and comments that have improved the quality of the paper. The author is partially supported by MPNTR grant, no. 174017, Serbia.

\end{document}